\documentclass[a4paper,11pt,reqno]{amsart}
\usepackage{amsmath, amssymb, amsthm}
\usepackage{fullpage, hyperref}

\newcommand{\ZZ}{\mathbb{Z}}
\newcommand{\QQ}{\mathbb{Q}}
\newcommand{\RR}{\mathbb{R}}
\DeclareMathOperator{\sgn}{sgn}
\newcommand{\ipos}[1]{\sgn^+{(#1)}}

\newtheorem{thm}{Theorem}
\newtheorem{lemma}[thm]{Lemma}
\newtheorem{ques}[thm]{Question}
\newtheorem{corr}[thm]{Corollary}
\theoremstyle{remark}

\title{On the Open Problem in the 2023 Simon Marais Mathematics Competition}
\author[A.\ Pisani]{Anthony Pisani}
\address{School of Mathematics, Monash University, Clayton, Australia}
\email{anthony.pisani1@monash.edu}
\date{\today}

\begin{document}
	\begin{abstract}
		We prove that, for $n \in \ZZ^+$ not a perfect square,
		the only pairs of positive integers $\{a, b\}$
		for which there exists $r \in \RR$ such that $r^a + \sqrt{n}$
		and $r^b + \sqrt{n}$ are both rational numbers are those of
		the form $\{a, a\}$ or $\{g, 2g\}$ (with $g$ odd). This resolves
		an open question presented
		in the 2023 Simon Marais Mathematics Competition.
	\end{abstract}
	
	\maketitle

	\section{Introduction}\label{s:intro}

	The Simon Marais Mathematics Competition is
	an annual undergraduate mathematics competition, established in 2017
	(\cite{smmc}). In an unusual tradition for the format,
	one of the questions presented to participants is an open problem:
	full marks are awarded for the (known) solution to a given special case,
	with bonus marks for further progress. To the author's knowledge,
	none of the open problems set in the nine-year history of
	the competition have previously been resolved. \\

	The open problem on the 2023 competition asks the following:

	\begin{ques}[{\cite[Question B4(b)]{paper}}] \label{q:main}
		Let $n$ be a positive integer that is not a perfect square.
		Find all pairs $(a, b)$ of positive integers for which
		there exists a real number $r$, such that
		\[ r^a + \sqrt{n} \qquad \textrm{and} \qquad r^b + \sqrt{n} \]
		are both rational numbers.
	\end{ques}

	In the case where $r > 0$, which is highlighted
	in the competition paper, the official solutions \cite{sols} prove that
	the pairs sought in Question~\ref{q:main} are precisely those such that
	$a = b$. It is furthermore shown that, in general, distinct $a, b$ must
	have opposite parity, and that the valid pairs in which $a$ or $b$ is 1
	are exactly $\{ (1, 1), (1, 2), (2, 1) \}$. \\

	This paper aims to prove that the complete answer is
	as stated below.
	
	\begin{thm}\label{thm:main}
	The pairs $(a, b)$ sought in Question~\ref{q:main} are precisely
	those such that $a = b$ or
	$\{ a, b \} = \{ g, 2g \}$ for some odd positive integer $g$.
	\end{thm}

	Our proof, presented in \S\ref{s:proof}, hinges on a reduction
	to the case that $r \in \QQ(\sqrt{n})$, and
	the construction and study of a relevant Lucas sequence
	for the possibilities left open by the official solutions.
	The same techniques also naturally address some cases
	treated in \cite{sols} by other arguments, although
	the case where $r > 0$ (in a slightly generalised form)
	must still be handled with different ideas. \\

	We remark in closing that the official solutions to the 2023 paper
	do not provide attribution or any associated literature
	for Question~\ref{q:main}, as solutions for some other years do.
	Attempts by the author to locate a meaningful history
	in the research literature, both with and without LLM assistance, have
	been unsuccessful.

	\section{Acknowledgments and Usage of AI Tools}
	
	The author was supported
	by an Australian Government Research Training Program (RTP) Scholarship. \\

	The key ideas in this proof beyond those
	in the 2023 SMMC solutions \cite{sols} --- that is,
	the reduction and use of Lucas sequences alluded to
	in \S\ref{s:intro} --- were suggested by multiple LLMs,
	although the proof presented herein has been written
	entirely by hand. LLMs were also employed in an attempted
	literature search, as well as to identify (but not resolve)
	issues with the an early draft of this paper.

	\section{Preliminaries}\label{s:prelim}

	We write $\ipos{x}$ for $1$ if $x \ge 0$ and $-1$ otherwise,
	and $\bar{\cdot}$ for the unique non-trivial automorphism of
	the quadratic field extension $\QQ(\sqrt{n})$.
	As usual, $\nu_p (x)$ denotes the $p$-adic valuation
	of a rational number $x$, while $\lfloor x \rfloor$
	signifies the greatest integer less than or equal to $x$. \\
	
	Recall that a {\em Lucas sequence} is a sequence defined
	by a recurrence relation of the form
	\[ U_0 = 0, \qquad U_1 = 1, \qquad U_{k+2} = PU_{k+1} - QU_k
		\textrm{ for all } k \ge 0, \]
	where $P, Q$ are non-zero integers. The elements of
	such a sequence will be written $U_k (P, Q)$, or simply
	$U_k$ if $P$ and $Q$ are clear from context. A prime $p$ is
	{\em special} if it divides $\gcd{(P, Q)}$, and {\em regular}
	otherwise. The following results on $p$-adic valuations
	of $U_k$ will be needed in the sequel.

	\begin{thm}[{\cite[Theorem 6]{lucas}}] \label{thm:div-reg}
		Let $(U_k (P, Q))$ be a Lucas sequence and $p$ a regular prime.
		Then for all $k, \ell > 0$, we have
		$\nu_p(\gcd{(U_k, U_{\ell})}) = \nu_p(U_{\gcd{(k, \ell)}})$.
	\end{thm}

	\begin{thm}[{\cite[Theorem 46]{lucas}}] \label{thm:div-spec}
		Let $(U_k (P, Q))$ be a Lucas sequence and $p$ a special prime.
		Write $P' = P/p^{\nu_p(P)}$ and
		$Q' = Q/p^{\min{(\nu_p(Q), 2\nu_p(P))}}$.
		If $\nu_p(Q) \ge 2\nu_p(P)$, then
		$\nu_p (U_k) = (k-1) \nu_p(P) + \nu_p (U_k(P', Q'))$
		for all $k > 0$. If $2\nu_p(P) > \nu_p(Q)$, then for all $k > 0$
		\[ \nu_p (U_k) = \begin{cases}
			\lfloor k/2 \rfloor \nu_p(Q) & k \textrm{ odd}, \\
			\lfloor k/2 \rfloor \nu_p(Q) + (\nu_p (P) - \nu_p (Q)) +
				\nu_p(k/2) + h, & k \textrm{ even},
		\end{cases} \]
		where
		\[ h = \begin{cases}
			\nu_p(P'^2 - Q'), & p \in \{ 2, 3 \}, \,
				2\nu_p(P) = \nu_p(Q) + 1, \textrm{ and }
				p \mid k/2, \\
			0, & \textrm{otherwise}.
		\end{cases} \]
	\end{thm}

	(Note that the definition of $Q'$ is incorrect in \cite[Theorem 46]{lucas};
	a corrected statement comes from Theorems 43 and 45.)
	It will be convenient to extract a more targeted result
	from Theorem~\ref{thm:div-spec} for later use.

	\begin{corr}\label{corr:div-spec}
		Let $(U_k (P, Q))$ be a Lucas sequence and $p$ a special prime.
		If $\nu_p (Q) = 1$, then $\nu_p (U_k) \ge \lfloor k/2 \rfloor$
		for any $k > 0$, with equality for odd $k$.
	\end{corr}
	\begin{proof}
		Since $2\nu_p (P) > \nu_p (P) \ge \nu_p (Q)$,
		Theorem~\ref{thm:div-spec} yields
		$\nu_p (U_k) \ge \lfloor k/2 \rfloor \nu_p(Q) =
			\lfloor k/2 \rfloor$, with equality for odd $k$
		as claimed.
	\end{proof}

	In order to prove Theorem~\ref{thm:main}, lower bounds
	on the elements of $U_k$ will also be employed. The following
	elementary result suffices for our purposes. Its proof ---
	an easy induction --- is omitted.

	\begin{lemma}\label{lemma:luc-size}
		Let $(U_k (P, Q))$ be a Lucas sequence with $P > 0 > Q$.
		Then $U_k > k$ for $k > 5$, and also for $k > 3$
		if $P > 1$ or $Q < -1$.
	\end{lemma}
	
	\section{Proof of Theorem~\ref{thm:main}}\label{s:proof}

	Throughout this section, fix $n \in \ZZ^+$ which is
	not a perfect square.
	It is routine to confirm that the stated possibilities
	in Theorem~\ref{thm:main} do occur (taking $r = (1/2-\sqrt{n})^{1/g}$
	in accordance with \cite{sols} when $\{ a, b \} = \{ g, 2g \}$).
	In the reverse direction,
	it is therefore permissible to assume that $a \neq b$, or indeed ---
	without loss of generality --- that $a < b$. Our first lemma
	shows that still more stringent conditions may be imposed,
	which will prove crucial later on.

	\begin{lemma}\label{lemma:red}
		Let $a < b$ be positive integers, and $r \in \RR$
		be such that $x_a = r^a + \sqrt{n}$ and $x_b = r^b + \sqrt{n}$
		are rational. Writing $g = \gcd{(a, b)}$, define $\alpha = a/g$
		and $\beta = b/g$. Then $\alpha, \beta$ are coprime. There exists
		$s \in \QQ (\sqrt{n})$ such that $s^{\alpha} + \sqrt{n} = x_a$ and
		$s^{\beta} + \sqrt{n} = x_b$; moreover, if $g$ is even,
		$s\overline{s} \ge 0$.
	\end{lemma}
	\begin{proof}
		Set $s = r^g$. Clearly $\alpha, \beta$ are coprime and
		the stated equalities involving $s$ hold. Consequently,
		by B\'ezout's identity, there exist $k, \ell \in \ZZ$
		such that $\alpha k + \beta \ell = 1$: it follows that
		$s = 0 \in \QQ(\sqrt{n})$ or
		$s = (s^{\alpha})^k (s^{\beta})^{\ell} =
		(x_a - \sqrt{n})^k (x_b - \sqrt{n})^{\ell} \in \QQ(\sqrt{n})$.
		Finally, if $g$ is even, then
		$x_a = r^{\alpha g} + \sqrt{n} \ge \sqrt{n}$, and similarly
		for $x_b$. This makes both factors in the product
		$s = (x_a - \sqrt{n})^k (x_b - \sqrt{n})^{\ell}$ non-negative,
		with the same effect for its conjugate
		$\overline{s} = (x_a + \sqrt{n})^k (x_b + \sqrt{n})^{\ell}$,
		so $s\overline{s} \ge 0$.
	\end{proof}

	In light of Lemma~\ref{lemma:red}, it will be permissible to assume
	that $r \in \QQ(\sqrt{n})$ and $a, b$ are coprime where desired
	below. \\

	We now summarise most of the material required from the solutions
	to the competition paper \cite{sols} in the form of two results.
	The first records some straightforward observations, while the second
	handles one main family of possible solutions. Note that
	all statements have been reformulated to better fit
	the argument herein.
	
	\begin{lemma}[cf.\ {\cite{sols}}] \label{lemma:pre-base}
		Let $a < b$ be positive integers, and $r \in \QQ (\sqrt{n})$
		be such that $x_a = r^a + \sqrt{n}$ and $x_b = r^b + \sqrt{n}$
		are rational. Then the following hold:
		\begin{enumerate}
			\item \label{eq:obs1} $(r\overline{r})^c = x_c^2 - n$
				for $c \in \{a, b\}$; and
			\item \label{eq:obs2}
				$(x_a \pm \sqrt{n})^b = (x_b \pm \sqrt{n})^a$.
		\end{enumerate}
	\end{lemma}
	\begin{proof}
		For $c \in \{ a, b \}$, we have $r^c = x_c - \sqrt{n}$,
		so that
		$\overline{r}^c = \overline{r^c} = x_c + \sqrt{n}$
		(noting $r \in \QQ(\sqrt{n})$ and $x_c \in \QQ$).
		This yields
		$x_c^2 - n = r^c \overline{r}^c = (r\overline{r})^c$,
		which is Item~\ref{eq:obs1}. For Item~\ref{eq:obs2},
		we have $(x_a + \sqrt{n})^b = \overline{r}^{ab} =
			(x_b + \sqrt{n})^a$, with a similar argument
		for the subtractive case using $r$ in place of $\overline{r}$.
	\end{proof}

	\begin{thm}[cf.\ {\cite{sols}}] \label{thm:base}
		Retain the hypotheses from Lemma~\ref{lemma:pre-base}.
		Then $r\overline{r} \le 0$.
	\end{thm}
	\begin{proof}
		Recall $b > a$.
		Suppose for the sake of contradiction that $r\overline{r} > 0$.
		Item~\ref{eq:obs1} of Lemma~\ref{lemma:pre-base} then implies
		$|x_b|, |x_a| > \sqrt{n}$. In particular, the function
		\[ f(x) = \ln{\left\vert \frac{(x_a + x)^b}{(x_b + x)^a} \right\vert} =
			b\ln{|x_a + x|} - a\ln{|x_b + x|} \]
		is defined on the interval $[-\sqrt{n}, \sqrt{n}]$. 
		By Item~\ref{eq:obs2} of the lemma,
		$f(-\sqrt{n}) = f(\sqrt{n}) = 0$. The same item demands
		$|x_a| < |x_b|$, for otherwise
		$ |(x_b + \ipos{x_a} \sqrt{n})^a| \le (|x_b| + \sqrt{n})^a
			< (|x_a| + \sqrt{n})^b =
			|(x_a + \ipos{x_a} \sqrt{n})^b|$.
		On the other hand, $f$ is differentiable on $[ -\sqrt{n}, \sqrt{n} ]$
		with derivative $f'(x) = b/(x_a + x) - a/(x_b + x)$. A value of
		$x$ at which this vanishes must satisfy
		$0 = b(x_b + x) - a(x_a + x) = (bx_b - ax_a) + (b-a)x$.
		Therefore, using the reverse triangle inequality
		and the previous inequalities in this paragraph,
		\[ |x| = \left\vert \frac{bx_b - ax_a}{b-a} \right\vert
			\ge \frac{||bx_b| - |ax_a||}{|b-a|} =
			\frac{b|x_b| - a|x_a|}{b-a} > |x_a| >
			\sqrt{n}. \]
		It follows that $f$ is strictly increasing or decreasing
		on $[ -\sqrt{n}, \sqrt{n} ]$, contradicting
		the requirement that $f(-\sqrt{n}) = 0 = f(\sqrt{n})$.
	\end{proof}

	Before proceeding to the second part of the argument, it will also
	be necessary to eliminate some degenerate cases for $r$, and to place
	further conditions on $a$.

	\begin{lemma}\label{lemma:ext-base}
		Retain the hypotheses from Lemma~\ref{lemma:pre-base}.
		Then $r\overline{r} < 0$ and $r + \overline{r} \neq 0$.
		Furthermore, if $a, b$ are coprime, $a$ must be odd.
	\end{lemma}
	\begin{proof}
		By Theorem~\ref{thm:base}, $r\overline{r} \le 0$.
		If $r \overline{r} = 0$, then $r = 0$, so that
		$r^a + \sqrt{n} = \sqrt{n}$ is irrational. Hence,
		$r \overline{r} < 0$. \\

		Now assume, for the sake of contradiction, that
		$r + \overline{r} = 0$. Thus $r \in \QQ(\sqrt{n})$
		has the form $r = v\sqrt{n}$ for some $v \in \QQ$.
		Consequently, for $c \in \{a, b\}$,
		$x_c = r^c + \sqrt{n} = v^c \sqrt{n}^c + \sqrt{n}$.
		Since $x_c$ and $v$ are rational, this can only hold if
		$c$ is odd; moreover, $v^c n^{(c-1)/2}$ must equal $-1$,
		lest $\sqrt{n} = x_c/(1+v^cn^{(c-1)/2})$. Therefore
		\[ 1 = \frac{v^b n^{(b-1)/2}}{v^a n^{(a-1)/2}} = (v \sqrt{n})^{b-a}. \]
		As $v \sqrt{n} \notin \QQ$ and $b > a$, this is again impossible. \\

		Finally, suppose that $a, b$ are coprime with $a$ even.
		Thus, $b > a$ is odd.
		Item~\ref{eq:obs1} of Lemma~\ref{lemma:pre-base} and
		the above conclusion that $r \overline{r} < 0$
		therefore imply $|x_b| < \sqrt{n} < |x_a|$; however,
		this yields
		$|(x_b + \ipos{x_a} \sqrt{n})^a| \le (|x_b| + |\ipos{x_a} \sqrt{n}|)^a
			< (2 \sqrt{n})^a \le (2 \sqrt{n})^b
			< |(x_a + \ipos{x_a} \sqrt{n})^b|$,
		which contradicts Item~\ref{eq:obs2} of the lemma. Hence,
		$a$ is odd as claimed.
	\end{proof}

	Let us now investigate the case $r \overline{r} < 0$.
	The centrepiece of the argument in this case is the observation
	that, for any non-zero $M \in \RR$, the sequence defined by
	\[ U_k = M^{k-1} \left(
		\frac{\overline{r}^k - r^k}{\overline{r} - r}
	\right) = \frac{(M \overline{r})^k - (Mr)^k}{M \overline{r} - Mr}
	\qquad \textrm{for all } k \ge 0 \]
	(assuming $\overline{r} \neq r$) satisfies the recurrence relation
	\[ U_0 = 0, \qquad U_1 = 1, \qquad U_{k+2} = M(r + \overline{r}) U_{k+1}
		-M^2 r\overline{r} U_k
		\textrm{ for all } k \ge 0. \]
	If $M$ is chosen so that the coefficients on the right-hand side are
	non-zero integers, this makes $U_k$ a Lucas sequence. Using this fact,
	along with the conclusions about $U_k$ that can be drawn
	from the properties of $r^c$ and $\overline{r}^c$
	(for $c \in \{ a, b \}$), it will be possible to determine $a$ and $b$.
	We begin by selecting $M$ and recording some observations on $U_k$.
	
	\begin{lemma}\label{lemma:luc-base}
		Retain the hypotheses from Lemma~\ref{lemma:pre-base}.
		Write $r = n_1/d_1 - n_2\sqrt{n}/d_2$
		with each $n_i/d_i$ a fraction in lowest terms, and define
		$P = 2d_2 |n_1|/\lambda$ and
		$Q = (n_1^2 d_2^2 - nn_2^2 d_1^2)/\lambda^2$,
		where $\lambda \in \ZZ^+$ is maximal\footnote{
			:~For completeness, set $\lambda = 1$
			if $P = Q = 0$, although it will be proved
			this does not occur anyway.
		} with respect
		to the property that $P$ and $Q$ are integers.
		Then $P > 0 > Q$.  Furthermore, for all $k \ge 0$,
		the Lucas sequence $(U_k (P,Q))$ satisfies
		\[ U_k = \left(
			\frac{d_1 d_2 \ipos{n_1}}{\lambda}
		\right)^{k-1} \left(
			\frac{\overline{r}^k - r^k}{\overline{r} - r}
		\right). \]
	\end{lemma}
	\begin{proof}
		Writing $M = d_1 d_2 \ipos{n_1}/\lambda$, it is
		easily verified that $P = M(r + \overline{r})$ and
		$Q = M^2 r\overline{r}$. The first two claims
		of Lemma~\ref{lemma:ext-base} therefore establish that
		$P \neq 0$ and $Q < 0$; since $P = 2d_2 |n_1|/\lambda \ge 0$,
		this completes the first part of the proof. The remainder
		follows from the remarks preceding the lemma:
		$\overline{r} \neq r$, for otherwise $r\overline{r} = r^2$
		would be non-negative, contrary to Lemma~\ref{lemma:ext-base}.
	\end{proof}

	The substance of the argument is now presented in a theorem and
	a preliminary lemma. The latter depends crucially on an additional
	assumption that $a > 1$, leaving $a = 1$ to be treated
	separately in the body of the theorem. The conclusion of
	Lemma~\ref{lemma:ext-base} that $a$ is odd for $a, b$ coprime
	also plays a role in both proofs, as does the fact that
	$\binom{b}{i} = b\binom{b-1}{i-1}/i$ noted in \cite{sols}.
	
	\begin{lemma}\label{lemma:div-b}
		Retain the hypotheses from Lemma~\ref{lemma:luc-base},
		including the definition of $U_k$. Additionally, suppose that
		$a, b$ are coprime with $a > 1$. Then each prime $p$ that
		divides $U_b$ also divides $n_2$. Moreover,
		$\nu_p (U_b) \le \nu_p(b)$ for every such prime, and so
		$U_b \mid b$.
	\end{lemma}
	\begin{proof}
		Write $M = d_1 d_2 \ipos{n_1} / \lambda \in \QQ$
		as in the proof of Lemma~\ref{lemma:luc-base} for brevity.
		Also, observe that for $c \in \{a, b\}$,
		
		\begin{align}
			U_c & = M^{c-1} \left(
				\frac{\overline{r}^c - r^c}{\overline{r} - r}
			\right) \nonumber
			= M^{c-1} \left(
				\frac{x_c + \sqrt{n} - (x_c - \sqrt{n})}{
					2n_2\sqrt{n}/d_2
				}
			\right) \nonumber \\
			& = M^{c-1} \left(
				\frac{2d_2\sqrt{n}}{2n_2\sqrt{n}}
			\right)
			= M^{c-1} d_2/n_2 \label{eq:Uc}
		\end{align}
				
		We first prove the claim that all prime divisors $p$ of $U_b$
		divide $n_2$, distinguishing two cases. \\

		If $p \mid \gcd{(P,Q)}$, then $\nu_p (Q) = 1$, lest
		$\lambda$ be replaceable by $\lambda p > \lambda$
		in the statement of Lemma~\ref{lemma:luc-base}.
		Corollary~\ref{corr:div-spec} hence implies that
		$\nu_p (U_k) \ge \lfloor k/2 \rfloor$,
		with equality for odd $k$. In particular, recalling
		from Lemma~\ref{lemma:ext-base} that coprimality of $a$ and
		$b > a$ requires odd $a$, $\nu_p (U_b) \ge
		\lfloor b/2 \rfloor > \lfloor a/2 \rfloor = \nu_p (U_a)$.
		It follows from Equation~\ref{eq:Uc} that $\nu_p (M) > 0$,
		and consequently that $\nu_p (n_2) = (a-1) \nu_p (M) +
			\nu_p(d_2) - \nu_p(U_a) \ge a - 1 - a/2 > 0$
		as sought. \\

		By contrast, if $p \nmid \gcd{(P,Q)}$, Theorem~\ref{thm:div-reg}
		and coprimality of $a, b$ yield $\nu_p (\gcd{(U_a, U_b)}) =
		\nu_p(U_{\gcd{(a, b)}}) = \nu_p(U_1) = 0$. Since $p$ divides
		$U_b$, this implies $\nu_p (U_a) = 0 < \nu_p(U_b)$; it again
		follows from Equation~\ref{eq:Uc} that $\nu_p (M) > 0$.
		As before, $\nu_p (n_2) = (a-1) \nu_p(M) + \nu_p(d_2) -
			\nu_p(U_a) \ge a-1 > 0$, the desired result. \\

		Having proved $p \mid n_2$, we proceed to demonstrate
		that $\nu_p (U_b) \le \nu_p(b)$. By Lemma~\ref{lemma:luc-base},

		\begin{align*}
			U_b & = \left(
				\frac{d_1 d_2 \ipos{n_1}}{\lambda}
			\right)^{b-1} \left(
				\frac{\overline{r}^b - r^b}{\overline{r} - r}
			\right) \\
			& = \left(
				\frac{\ipos{n_1}}{\lambda}
			\right)^{b-1} \left( \frac{
				(n_1 d_2 + n_2 d_1 \sqrt{n})^b -
					(n_1 d_2 - n_2 d_1 \sqrt{n})^b
			}{2n_2 d_1 \sqrt{n}} \right) \\
			& = \left(
				\frac{\ipos{n_1}}{\lambda}
			\right)^{b-1} \sum_{
				\substack{i=0 \\ i \textrm{ odd}}
			}^b \binom{b}{i} \frac{
				(n_1 d_2)^{b-i} (n_2 d_1 \sqrt{n})^i
			}{n_2 d_1 \sqrt{n}} \\
			& = \left(
				\frac{\ipos{n_1}}{\lambda}
			\right)^{b-1} \sum_{
				\substack{i=0 \\ i \textrm{ odd}}
			}^b \binom{b}{i} (n_1 d_2)^{b-i}
				(n_2 d_1)^{i-1} n^{(i-1)/2}.
		\end{align*}

		Now, each $n_i$ is coprime to the corresponding $d_i$
		by construction, so that $p \nmid d_2$; manipulation of
		Equation~\ref{eq:Uc} (with $c = b$) in turn reveals
		\[0 < (b-1) \nu_p (\lambda) +
			\nu_p (U_b n_2) = (b-1) \nu_p (d_1 d_2 \ipos{n_1}) +
			\nu_p(d_2) = (b-1) \nu_p (d_1), \]
		from which it follows
		that $p \mid d_1$ and hence $p \nmid n_1$. Each term
		in the sum on the right-hand side of the above equation
		therefore has a $p$-adic valuation of
		$\nu_p (\binom{b}{i} (n_1 d_2)^{b-i}
			(n_2 d_1)^{i-1} n^{(i-1)/2}) =
			\nu_p( b \binom{b-1}{i-1}/i ) +
			(i-1) \nu_p(n_2 d_1) + (i-1) \nu_p(n)/2 \ge
			\nu_p(b) - \nu_p(i) + i-1 \ge \nu_p (b)$,
		with equality precisely when $i = 1$. The sum itself
		therefore has a $p$-adic valuation of $\nu_p (b)$.
		Consequently, $\nu_p (U_b) = \nu_p (b) -
			(b-1) \nu_p (\lambda) \le \nu_p (b)$
		as claimed; the remainder of the lemma is immediate.
	\end{proof}

	\begin{thm}\label{thm:only-12}
		Retain the hypotheses from Lemma~\ref{lemma:luc-base},
		including the definition of $U_k$, and suppose additionally that
		$a, b$ are coprime. Then $a = 1$ and $b = 2$.
	\end{thm}
	\begin{proof}
		Suppose firstly that $a > 1$. Lemma~\ref{lemma:ext-base}
		and coprimality of $a, b$ force $a$ to be odd, implying
		$a \ge 3$ and thus $b \ge 4$. It follows
		by Lemma~\ref{lemma:luc-size} that either $U_b (P, Q) > b$
		or $P = -Q = 1$ and $b \in \{ 4, 5 \}$. The former contradicts
		the conclusion $U_b \mid b$ of Lemma~\ref{lemma:div-b},
		as does $P = -Q = 1$ and $b = 4$, so that $b = 5$ and
		$a = 3$ with $P = -Q = 1$ is the only possibility.
		In this situation, however, Equation~\ref{eq:Uc} of
		Lemma~\ref{lemma:div-b} provides the contradiction
		\[ \sqrt{\frac{5}{2}} = \left( \frac{U_5}{U_3} \right)^{1/(5-3)} =
			\left\vert \frac{d_1 d_2 \ipos{n_1}}{\lambda} \right\vert
			\in \QQ. \]

		Hence $a = 1$ and $b > a = 1$. To establish
		that $b = 2$, it will suffice to employ a variant of
		the argument used in the second half of the proof of
		Lemma~\ref{lemma:div-b}. Per Item~\ref{eq:obs2} of
		Lemma~\ref{lemma:pre-base}, we have
		\[ 1 = \frac{(x_b + \sqrt{n}) - (x_b - \sqrt{n})}{2\sqrt{n}} =
			\frac{(x_a + \sqrt{n})^b -
				(x_a - \sqrt{n})^b}{2\sqrt{n}} = \sum_{
				\substack{i=0 \\ i \textrm{ odd}}
			}^b \binom{b}{i} x_a^{b-i} n^{(i-1)/2}
		\]
		with $x_a, x_b \in \QQ$. The fact that $i$ is odd and $n > 0$
		means each term of this sum has the same sign, and the magnitude
		of the term with $i = 1$ is $|bx_a^{b-1}|$. The above equality
		thus demands $|x_a| < 1$; in particular, either $x_a = 0$ ---
		contrary to the result of Lemma~\ref{lemma:ext-base} that
		$r + \overline{r} \neq 0$ --- or $\nu_p (x_a) < 0$
		for some prime $p$. For such $p$, each term in the sum has
		a $p$-adic valuation of
		$\nu_p (\binom{b}{i} x_a^{b-i}
			n^{(i-1)/2}) =
			\nu_p( b \binom{b-1}{i-1}/i ) +
			(b-i) \nu_p(x_a) + (i-1) \nu_p(n)/2 \ge
			\nu_p(b) - \nu_p(i) + (i-1) \cdot {-\nu_p (x_a)} +
				(b-1)\nu_p (x_a) \ge
			\nu_p (b) + (b-1) \nu_p(x_a)$,
		with equality precisely when $i = 1$. The $p$-adic valuation
		of the sum itself is therefore given by both
		$\nu_p (b) + (b-1) \nu_p(x_a)$ and $\nu_p (1) = 0$, so that
		$b \ge 2^{\nu_p (b)} = 2^{(b-1) \cdot {-\nu_p (x_a)}} \ge 2^{b-1}$;
		this holds only when $b \le 2$.
	\end{proof}

	Theorem~\ref{thm:main} now follows from Lemmas~\ref{lemma:red}
	and~\ref{lemma:ext-base}, and Theorem~\ref{thm:only-12}.

	\begin{small}
	
	\end{small}

\begin{thebibliography}{9}
	
	\bibitem{lucas}
	C.~J.--C.\ Ballot and H.~C.\ Williams.
	\newblock {\em The Lucas Sequences: Theory and Applications}.
	\newblock CMS/CAIMS Books in Mathematics, Springer Nature, Switzerland, 2023.

	\bibitem{paper}
	Simon Marais Mathematics Competition.
	\newblock {\em 2023 Paper B}, 2023. \\
	\newblock \url{https://www.simonmarais.org/uploads/8/2/3/5/82358688/smmc2023_paper_b.pdf}

	\bibitem{sols}
	Simon Marais Mathematics Competition.
	\newblock {\em 2023 Solutions}, pp.\ 14--16, 2023. \\
	\newblock \url{https://www.simonmarais.org/uploads/8/2/3/5/82358688/smmc2023solutions.pdf}

	\bibitem{smmc}
	Simon Marais Mathematics Competition.
	\newblock {\em Background}.
	\newblock Online, accessed 2026. \\
	\newblock \url{https://www.simonmarais.org/background.html}

	\end{thebibliography}
\end{document}